\documentclass[12pt]{amsart}
\usepackage{amssymb}
\newtheorem{thm}{Theorem}
\newtheorem{lem}{Lemma}
\newtheorem{cor}{Corollary}
\DeclareSymbolFont{script}{U}{eus}{m}{n}
\DeclareMathSymbol{\Wedge}{0}{script}{"5E}

\newcommand{\oo}[2]{\begin{picture}(24,11)
\put(4,1.5){\makebox(0,0){$\bullet$}}
\put(20,1.5){\makebox(0,0){$\bullet$}}
\put(4,1.5){\line(1,0){16}}
\put(4,8){\makebox(0,0){$\scriptstyle #1$}}
\put(20,8){\makebox(0,0){$\scriptstyle #2$}}
\end{picture}}
\title[Metrisability]
{Metrisability of three-dimensional projective structures}
\author[Michael Eastwood]{Michael Eastwood}
\address{\hskip-\parindent
School of Mathematical Sciences,
University of Adelaide,\newline 
SA 5005, Australia}
\email{meastwoo@member.ams.org}
\subjclass{53A20}

\thanks{This work was supported by the Simons Foundation grant 346300 and the
Polish Government MNiSW 2015--2019 matching fund. It was carried out whilst the
author was at the Banach Centre at IMPAN in Warsaw for the Simons Semester
`Symmetry and Geometric Structures.' I would also like to thank,
firstly, Katharina Neusser, who was a Simons Professor during this semester,
for many useful remarks concerning the metrisability problem, secondly, Maciej
Dunajski for pointing out that my original proof of Theorem~\ref{mobility} was
incomplete, and, thirdly, Felipe Contatto and Maciej Dunajski for many useful 
remarks concerning Theorem~\ref{degenerate_bound}.}

\begin{document}\raggedbottom
\begin{abstract} 
We solve the metrisability problem for generic three-dimensional projective
structures.
\end{abstract}    
\maketitle 

\setcounter{section}{-1}
\section{Introduction}
Let $M$ be a connected smooth oriented three-dimensional manifold. Suppose
$\nabla_a:TM\to\Wedge^1\otimes TM$ is a torsion-free connection. We shall say
that a symmetric covariant $2$-tensor $g_{ab}$ is a {\em metric\/} if and only
if it is non-degenerate (irrespective of its signature). In this article we
find necessary and sufficient local conditions, for a generic
connection~$\nabla_a$ (precisely in the sense of Corollary~\ref{punchline}), in
order that there be a metric $g_{ab}$ whose geodesics coincide with the
geodesics of $\nabla_a$ as unparameterised curves. The two-dimensional case was
solved in~\cite{BDE}. The three-dimensional case has an entirely different
character. Obstructions to metrisability in the three-dimensional case were
found in~\cite{DE} to which we refer for notational details and background here
omitted. As explained in~\cite{DE}, we may always fix a volume form
$\epsilon_{bcd}$ and normalise $\nabla_a$ so that $\nabla_a\epsilon_{bcd}=0$.
Henceforth, we shall suppose this has been done. We shall refer to connections
with the same unparameterised geodesics as {\em projectively equivalent}.

The following theorem was proved in the two-dimensional case by
R.~Liouville~\cite{L} and in general by Mike\v{s}~\cite{M} following
observations of Sinjukov~\cite{S} (see also~\cite{EM}). In three dimensions we 
have:

\begin{thm} For the existence of a metric whose Levi-Civita connection is
projectively equivalent to~$\nabla_a$, it is necessary and sufficient to have a
non-degenerate symmetric tensor $\sigma^{ab}$ such that
\begin{equation}\label{metrisability}\textstyle
(\nabla_a\sigma^{bc})_\circ\equiv\nabla_a\sigma^{bc}
-\frac12\delta_a{}^{(b}\nabla_d\sigma^{c)d}=0.
\end{equation}
\end{thm}
This `metrisability equation' may be investigated by prolongation and it
emerges that if $\sigma^{bc}$ satisfies~(\ref{metrisability}),
then~\cite[Equation~(13)]{DE}
\begin{equation}\label{Vsigma}V^{(ab}{}_d\sigma^{c)d}=0,\end{equation}
where $V^{ab}{}_c=V^{(ab)}{}_c$ is the trace-free tensor defined by
\begin{equation}\label{V}2\epsilon^{de(a}\nabla_d\nabla_eX^{b)}=V^{ab}{}_cX^c
\end{equation}
and $\epsilon^{abc}$ is dual to $\epsilon_{abc}$ (normalised such that
$\epsilon^{abc}\epsilon_{abc}=6$, for example). The tensor $V^{ab}{}_c$ is
equivalent to the usual projectively invariant {\em Weyl tensor\/} and, as is
shown in~\cite[Theorem~1.3]{DE}, if $\sigma^{bc}$ is non-degenerate and 
(\ref{Vsigma}) holds, then the projectively invariant tensor
\begin{equation}\label{Q}
Q_{ab}{}^c\equiv\epsilon_{pq(a}V^{pr}{}_{b)}V^{qc}{}_r\end{equation}
must vanish. 

The vanishing of $Q_{ab}{}^c$ is the primary obstruction to metrisability and 
in this article we shall show that in this case, under some mild genericity 
assumptions on the projective structure, we can find further necessary 
conditions that are sufficient for a complete solution to the metrisability 
problem.

It is useful to make some preliminary observations on the solution space to the
metrisabilty equation (\ref{metrisability}) as follows.
\begin{lem}\label{algebraic_lemma} Suppose $\rho^{bc}$ and $\sigma^{bc}$ are
symmetric $3\times 3$ matrices and that
\begin{equation}\label{pure_trace}
f_a\rho^{bc}+h_a\sigma^{bc}=\delta_a{}^{(b}\kappa^{c)}\end{equation}
for some covectors $f_a$ and $h_a$, and some vector $\kappa^c$. Then 
$\kappa^c=0$. 
\end{lem}
\begin{proof} Without loss of generality, choose a frame so that $f_3=h_3=0$.
Then (\ref{pure_trace}) implies that $0=\delta_3{}^{(3}\kappa^{c)}$, from which
$\kappa^c=0$ is immediate.
\end{proof}
\begin{lem}\label{constancy} Suppose $\rho^{bc}$ and $\sigma^{bc}$ are linearly
independent solutions to~\eqref{metrisability}. If $f\rho^{bc}+h\sigma^{bc}$ is
also a solution, then $f$ and $h$ must be constant.
\end{lem}
\begin{proof} Firstly, note that $\rho^{bc}$ and $\sigma^{bc}$ cannot be
proportional on an open set since, if $h\sigma^{bc}$
solves~(\ref{metrisability}), then $\big((\nabla_ah)\sigma^{bc}\big)_\circ=0$
and Lemma~\ref{algebraic_lemma} implies that $h$ is locally constant
(cf.~\cite{EPS}). Hence, we may suppose, without loss of generality, that
$\rho^{bc}$ and $\sigma^{bc}$ are pointwise linearly independent (since this is
true on an open subset of $M$ and, as soon as $f$ and $h$ are constant on an
open subset, then they are everywhere constant by prolongation (see 
\cite[Theorem~3.1]{EM})). Substituting $f\rho^{bc}+h\sigma^{bc}$ into 
(\ref{metrisability}) gives
$$\big((\nabla_af)\rho^{bc}\big)_\circ
+\big((\nabla_ah)\sigma^{bc}\big)_\circ=0.$$
It follows from Lemma~\ref{algebraic_lemma} that 
$$(\nabla_af)\rho^{bc}+(\nabla_ah)\sigma^{bc}=0$$
and linear independence forces $\nabla_af=\nabla_ah=0$, as required.
\end{proof}

\section{The Pl\"ucker relations}
Suppose ${\mathbb{W}}$ is a $6$-dimensional real vector space and
$V\in\Wedge^2{\mathbb{W}}$. We shall say that $V$ is {\em simple\/} if and only
if $V=\rho\wedge\sigma$ for some $\rho,\sigma\in{\mathbb{W}}$. Consider
$V\wedge V\in\Wedge^4{\mathbb{W}}$, which evidently vanishes if $V$ is simple.
The converse is shown in~\cite{EMi}:
\begin{thm}\label{strong_pluecker}
The tensor $V$ is simple if and only if $V\wedge V=0$.
\end{thm}
We may employ the theory of highest weights to classify the irreducible 
representations of ${\mathrm{SL}}(3,{\mathbb{R}})$. Specifically, we may 
follow the notation of \cite{beastwood} in writing the general such 
representation as
$$\oo{a}{b}\quad\mbox{for}\enskip a,b\in{\mathbb{Z}}_{\geq 0},$$
with $\oo{0}{1}$ being the defining representation on~${\mathbb{R}}^3$. Then 
$$\textstyle\bigodot^2{\mathbb{R}}^3=\oo{0}{2}\qquad
\Wedge^2\oo{0}{2}=\oo{1}{2}\qquad\Wedge^4\oo{0}{2}=\oo{2}{1}$$
and we may read (\ref{Q}) as $Q=V\wedge V$ for $V\in\oo{1}{2}$. {From} 
Theorem~\ref{strong_pluecker} we obtain the following.
\begin{lem}\label{key} With $Q_{ab}{}^c$ as in~\eqref{Q}, we have 
\begin{equation}\label{Vsimple}
Q_{ab}{}^c=0\iff V^{ab}{}_c=\rho^{d(a}\sigma^{b)e}\epsilon_{cde}.\end{equation}
\end{lem}
\begin{proof} It suffices to recognise the right hand side of (\ref{Vsimple})
as $V=\rho\wedge\sigma$, which is immediate by Schur's lemma as soon as it is
non-zero.
\end{proof}
Indeed, a route to the vanishing of $Q_{ab}{}^c$ in~\cite{DE} was to show that
if (\ref{Vsigma}) holds and $\sigma^{ab}$ is invertible, then
\begin{equation}\label{Vrhosigma}
V^{ab}{}_c=\rho^{d(a}\sigma^{b)e}\epsilon_{cde}\end{equation}
for some symmetric contravariant $2$-tensor~$\rho^{ab}$.

In particular, if we define the {\em degree of mobility\/} of a metric to be
the dimension of the solution space of its associated metrisability equation,
then we obtain an immediate proof of the following theorem (due to Kiosak and
Mike\v{s}~\cite{K,KM}).

\begin{thm}\label{mobility} The degree of mobility of a three-dimensional
metric can only be $1$, $2$, or $10$.
\end{thm}
\begin{proof}
If $V_{ab}{}_c$ vanishes identically, then the structure is `projectively flat'
and prolongation shows that solution space of (\ref{metrisability}) may be
identified with $\oo{0}{2}\oplus\oo{0}{1}\oplus\oo{0}{0}$, which has dimension
$6+3+1=10$. Otherwise, the projectively invariant Weyl tensor $V^{ab}{}_c$ 
determines, by means of~(\ref{Vrhosigma}),
$${\mathrm{span}}\{\rho^{ab},\sigma^{ab}\}$$
from which any solution of the metrisability equation must be taken.
Lemma~\ref{constancy} completes the proof.
\end{proof}

Although not strictly relevant to the metrisability problem, similar arguments
bound the dimension of the solution space to (\ref{metrisability}) without
supposing a non-degenerate solution. These arguments yield
Theorem~\ref{degenerate_bound} below and, for its proof, we shall need the
following Lemmata. 
\begin{lem}\label{Loewy_Radwan} Suppose $W$ is a linear subspace of\/
$\bigodot^2{\mathbb{R}}^3$ consisting entirely of degenerate matrices. Then
$\dim W\leq 3$. Suppose $\dim W=3$. Then, generically, we can find a basis for
${\mathbb{R}}^3$ such that
\begin{equation}\tag{Type $W_2$}
W=\left\{\mbox{\footnotesize$\left[\begin{array}{ccc}
p&q&r\\ q&0&0\\ r&0&0\end{array}\right]$}\right\}\end{equation}
and, otherwise, find a basis so that
\begin{equation}\tag{Type $W_1$}
W=\left\{\mbox{\footnotesize$\left[\begin{array}{ccc}
p&q&0\\ q&r&0\\ 0&0&0\end{array}\right]$}\right\}.\end{equation}
\end{lem}
\begin{proof} This result was shown by Loewy and Radwan~\cite{LR} to whom the 
terminology $W_1$ and $W_2$ is also due.
\end{proof}

\begin{lem}\label{typeW1} In the terminology of the previous lemma, suppose 
$$\textstyle\rho^{bc},\sigma^{bc},\tau^{bc}
\in W_1\subset\bigodot^2{\mathbb{R}}^3$$ 
and that
\begin{equation}\label{thrice_pure_trace}
f_a\rho^{bc}+h_a\sigma^{bc}+g_a\tau^{bc}
=\delta_a{}^{(b}\kappa^{c)}\end{equation}
for some covectors $f_a,h_a,g_a$, and some vector $\kappa^c$. Then 
$\kappa^c=0$. 
\end{lem}
\begin{proof}
Normalised as in Lemma~\ref{Loewy_Radwan}, it follows that
$0=\delta_3{}^{(3}\kappa^{c)}$, from which $\kappa^c=0$ is immediate.
\end{proof}

\begin{thm}\label{degenerate_bound} Suppose that the solution space to
\eqref{metrisability} consists entirely of degenerate tensors~$\sigma^{bc}$.
Then the dimension of this space is at most\/~$3$. 
\end{thm}
\begin{proof} If $\rho^{bc},\sigma^{bc},\tau^{bc}$ are three linearly
independent solutions, pointwise of type~$W_1$, then, by
Lemma~\ref{Loewy_Radwan}, any other solution must be of the form 
$f\rho^{bc}+h\sigma^{bc}+g\tau^{bc}$, in which case
$$\big((\nabla_af)\rho^{bc}\big)_\circ+\big((\nabla_ah)\sigma^{bc}\big)_\circ
+\big((\nabla_ag)\tau^{bc}\big)_\circ=0$$
and Lemma~\ref{typeW1} implies that $f,h,g$ are all constant. We are left with
the possibility that somewhere on $M$, and hence in an open subset, we have
three linearly independent solutions, pointwise of type~$W_2$. In this case, it
follows from Lemma~\ref{Loewy_Radwan} that, locally, there is a non-vanishing
vector field~$\theta^b$, uniquely determined up to scale, so that all solutions
to (\ref{metrisability}) have the form $\sigma^{bc}=\theta^{(b}\phi^{c)}$ for
some other vector field~$\phi^c$. Fixing such a field~$\theta^b$, it follows by
simple linear algebra that, wherever $\theta^b$ and $\phi^c$ are pointwise
linearly independent,
$$(\nabla_a\theta^b)_\circ=-\omega_a\theta^b\quad\mbox{and}\quad
(\nabla_a\phi^b)_\circ=\omega_a\phi^b$$
for some uniquely determined $1$-form~$\omega_a$. Differentiating 
$(\nabla_a\phi^b)_\circ=\omega_a\phi^b$ once more and decomposing the result 
into its irreducible parts, we find that, in particular,
$$V^{ab}{}_c\phi^c=F^a\phi^b+F^b\phi^a,$$
where $F^a\equiv\epsilon^{abc}\nabla_b\omega_c$. Therefore, if there are three
linearly independent solutions of type~$W_2$, then
$$V^{ab}{}_c=F^a\delta_c{}^b+F^b\delta_c{}^a$$
and tracing over ${}^b{}_c$ shows that $V^{ab}{}_c=0$, in which case the
connection $\nabla_a$ is projectively flat and we have reached a contradiction.
\end{proof}

In fact, the proof of Theorem~\ref{degenerate_bound} shows that the only way
that the dimension of the solution space to (\ref{metrisability}) can be $3$ is
if all solutions are locally of type~$W_1$. It is shown in \cite{DE} that this
possibility is realised, by both the Egorov projective structure~\cite{E} and
also by another family of structures, the so-called `Newtonian' projective
structures. Contatto~\cite{C} has recently extended the Newtonian structures to
all higher dimensions~$n$, showing that, for these structures, all solutions to
the metrisability equation are degenerate and that the dimension of this space
is $n(n-1)/2$. Contatto conjectures that this is the maximal dimension for
degenerate solutions and proves this under the assumption that there is a
non-zero $1$-form $\omega_c$ such that $\sigma^{bc}\omega_c=0$ for all
solutions $\sigma^{bc}$ (in other words, that this space is pointwise
Loewy-Radwan~\cite{LR} type~$W_1$). (By contrast, if there is a non-degenerate 
solution, then the submaximal dimension of the solution space is 
smaller~\cite{FM,K,M}, namely $n(n-1)/2-(n-2)$.)

\section{Local solutions}
As already stated, the aim of this article is to find {\em local\/} solutions
to the metrisability equation~(\ref{metrisability}). The proof of
Theorem~\ref{mobility} provides a good illustration of how this works since it
is only where the Weyl tensor $V^{ab}{}_c$ is non-zero that (\ref{Vrhosigma})
determines the pencil
\begin{equation}\label{pencil}f\rho^{ab}+h\sigma^{ab}\end{equation}
of possible solutions to (\ref{metrisability}) for smooth functions $f$
and~$h$. If the connection $\nabla_a$ is projectively flat, then it is locally 
metrisable. Otherwise, we may restrict to the open subset of $M$ where 
$V^{ab}{}_c$ is non-zero. We shall have occasion to make similar restrictions 
concerning other manifestly open conditions and shall usually do so without 
comment.

\section{Extracting the scale of a solution}
The plan, in the remainder of this article, is to consider the pencil
(\ref{pencil}) determined by $V^{ab}{}_c$ and, where possible, extract from it 
a non-degenerate solution to the metrisability equation~(\ref{metrisability}). 
The ultimate step in such an extraction is to pin down the scale of a solution 
as follows (cf.~\cite{MT}).

\begin{thm}\label{scale} Suppose $\sigma^{bc}\in\Gamma(\oo{0}{2})$ is
non-degenerate. In order that $h\sigma^{bc}$ solve \eqref{metrisability} for
some smooth non-vanishing function $h$, it is firstly necessary that the
$1$-form $\omega_a\equiv\sigma_{bc}(\nabla_a\sigma^{bc})_\circ$ be exact, where
$\sigma_{bc}$ denotes the inverse of~$\sigma^{bc}$. Secondly, it is both
necessary and sufficient that
\begin{equation}\label{nec_and_suff}\textstyle
(\nabla_a\sigma^{bc}-\frac25\omega_a\sigma^{bc})_\circ=0.
\end{equation}
\end{thm}
\begin{proof} Let us note the following identity:
\begin{equation}\label{identity}\textstyle
\sigma_{bc}(\theta_a\sigma^{bc})_\circ
=\sigma_{bc}(\theta_a\sigma^{bc}-\frac12\delta_a{}^{(b}\theta_d\sigma^{c)d})
=\frac52\theta_a.
\end{equation}
Therefore, in case that $h\sigma^{ab}$ solves (\ref{metrisability}), 
$$(\nabla_a(h\sigma^{bc}))_\circ
=((\nabla_ah)\sigma^{bc})_\circ+h(\nabla_a\sigma^{bc})_\circ=0$$
implies that
\begin{equation}\label{dlogh}
\textstyle\frac52\nabla_ah+h\sigma_{bc}(\nabla_a\sigma^{bc})_\circ=0
\end{equation}
so $\omega_a=-\frac25\nabla_a\log h$ and is, therefore, exact. Substituting
back into (\ref{metrisability}) gives~(\ref{nec_and_suff}), as required.
\end{proof}

\section{Testing non-degeneracy}\label{mobilitytwo}
Assuming that the projective Weyl tensor $V^{ab}{}_c$ is non-zero and yet
$Q_{ab}{}^c=0$, as we may, Lemma~\ref{key} implies that 
$V^{ab}{}_c=\rho^{d(a}\sigma^{b)e}\epsilon_{cde}$ and that any such tensors 
$\rho^{bc}$ and $\sigma^{bc}$ are linearly independent. We can construct an
arbitrary linear combination of $\rho^{bc}$ and $\sigma^{bc}$ by the formula
\begin{equation}\label{formula}2T_{ad}V^{a(b}{}_e\epsilon^{c)de}
=T_{ad}\sigma^{ad}\rho^{bc}-T_{ad}\rho^{ad}\sigma^{bc}
=f\rho^{bc}+h\sigma^{bc}\end{equation}
for an arbitrary symmetric $2$-tensor $T_{ab}$. It may happen that all such
linear combinations are degenerate. If so, the projective structure defined by
$\nabla_a$ cannot be metrisable. Indeed, some examples of this phenomenon are
given in~\cite{DE}, especially the Egorov projective structure~\cite{E}. We 
may eliminate this possibility as follows.
\begin{thm}\label{generic}
In order that the projective structure defined by $\nabla_a$ with Weyl
curvature $V^{ab}{}_c$ be metrisable it is necessary that the composition
$$\begin{array}{rcl}\Gamma(\oo{2}{0})\ni T_{ab}
&\mapsto&T_{ad}V^{a(b}{}_e\epsilon^{c)de}
\in\Gamma(\oo{0}{2})\\
&&\qquad\begin{picture}(0,0)
\put(0,8){\rule{6pt}{.5pt}}
\put(0,0){$\downarrow$}
\end{picture}\\
&&\det(T_{ad}V^{a(b}{}_e\epsilon^{c)de}),
\end{array}$$
where
$\det(\tau^{bc})\equiv\tau^{ab}\tau^{cd}\tau^{ef}\epsilon_{ace}\epsilon_{bdf}$,
not vanish identically.
\end{thm}
With Theorem~\ref{generic} in place, we may use the formula (\ref{formula}) to
construct a rank $2$ sub-bundle of $\oo{0}{2}$ with the property that any
solution of (\ref{metrisability}) is necessarily a section of this bundle. It
remains to devise a test to determine whether a section
$f\rho^{bc}+h\sigma^{bc}$ of this sub-bundle can solve the metrisability
equation~(\ref{metrisability}). If, for some reason, we are reduced to sections
$h\sigma^{bc}$ of a non-degenerate line sub-bundle, then Theorem~\ref{scale} 
applies.

\section{Testing for solutions}
As in \S\ref{mobilitytwo}, we may construct from~$V^{ab}{}_c$, in accordance
with~(\ref{Vrhosigma}), two linearly independent symmetric contravariant
$2$-tensors $\rho^{bc}$ and $\sigma^{bc}$ with $\sigma^{bc}$ non-degenerate and
be assured that the solution to (\ref{metrisability}) that we seek is
necessarily from the pencil~(\ref{pencil}), i.e.~of the form
$f\rho^{bc}+h\sigma^{bc}$. It is shown by linear algebra in
Appendix~\ref{normal} that, in the generic case, we may suppose, without loss
of generality, that
\begin{equation}\label{WLG}
\begin{array}{l}
\bullet\enskip \sigma^{bc}\mbox{ is invertible},\\
\bullet\enskip \sigma_{bc}\rho^{bc}=0,\\[1pt]
\bullet\enskip \rho^{bc}\xi_b=0,\mbox{ for some smooth }\xi_b\not=0.
\end{array}
\end{equation}
Substituting $f\rho^{bc}+h\sigma^{bc}$ into~(\ref{metrisability}), we obtain
\begin{equation}\label{obtained}
((\nabla_af)\rho^{bc})_\circ+((\nabla_ah)\sigma^{bc})_\circ
+f(\nabla_a\rho^{bc})_\circ+h(\nabla_a\sigma^{bc})_\circ=0.\end{equation}
Set $\xi^a\equiv\sigma^{ab}\xi_b$. Contracting $\xi^a\sigma_{bc}$ into 
(\ref{obtained}) yields
$$\textstyle\frac52\xi^a\nabla_ah
+f\xi^a\sigma_{bc}(\nabla_a\rho^{bc})_\circ
+h\xi^a\sigma_{bc}(\nabla_a\sigma^{bc})_\circ=0.$$
Contracting $\xi^a\xi_b\xi_c$ into (\ref{obtained}) yields
$$\textstyle\frac12\xi^b\xi_b\xi^a\nabla_ah
+f\xi^a\xi_b\xi_c(\nabla_a\rho^{bc})_\circ
+h\xi^a\xi_b\xi_c(\nabla_a\sigma^{bc})_\circ=0.$$
We may eliminate $\xi^a\nabla_ah$ from these equations to obtain
$$f\big(\xi^d\xi_d\sigma_{bc}-5\xi_b\xi_c\big)
\xi^a(\nabla_a\rho^{bc})_\circ
+h\big(\xi^d\xi_d\sigma_{bc}-5\xi_b\xi_c\big)
\xi^a(\nabla_a\sigma^{bc})_\circ=0.$$
Notice that the smooth functions
\begin{equation}\label{phi_and_psi}\begin{array}{rcl}\phi
&\equiv&\big(\xi^d\xi_d\sigma_{bc}-5\xi_b\xi_c\big)
\xi^a(\nabla_a\rho^{bc})_\circ\\[4pt]
\psi&\equiv&\big(\xi^d\xi_d\sigma_{bc}-5\xi_b\xi_c\big)
\xi^a(\nabla_a\sigma^{bc})_\circ\end{array}\end{equation}
are completely determined by the normal form~(\ref{WLG}). We have proved 
the following result concerning the pencil (\ref{pencil}) determined 
by~$V^{ab}{}_c$. 
\begin{thm} Suppose $\rho^{ab}$ and $\sigma^{ab}$ are in normal
form~\eqref{WLG}. In order that $f\rho^{bc}+h\sigma^{bc}$ satisfy the
metrisability equation, it is necessary that 
\begin{equation}\label{f_versus_h}
f\phi+h\psi=0,
\end{equation}
where $\phi$ and $\psi$ are canonically determined by \eqref{WLG} according 
to~\eqref{phi_and_psi}.
\end{thm}
\begin{cor}\label{punchline}
Wherever one of $\phi$ or $\psi$ defined by \eqref{phi_and_psi} is non-zero, we
may determine whether the projective structure is metrisable.
\end{cor}
\begin{proof}
On $\{\phi\not=0\}$ we may write $f=-h\psi/\phi$ and conclude that the 
purported solution of (\ref{metrisability}) has the form
\begin{equation}\label{purported}
h\big(\sigma^{ab}-(\psi/\phi)\rho^{ab}\big).\end{equation}
If $\sigma^{ab}-(\psi/\phi)\rho^{ab}$ is singular, then the projective 
structure cannot be metrisable. Otherwise, we may invoke Theorem~\ref{scale} 
to decide the matter. 

If $\phi$ vanishes identically, then $h\psi=0$ and on $\{\psi\not=0\}$ we
conclude that $h=0$. In this case our purported solution of 
(\ref{metrisability}) is $f\rho^{ab}$. This is singular so our projective 
structure is not metrisable.
\end{proof}
\begin{cor} We have solved the metrisability problem for generic 
three-dimensional projective structures.
\end{cor}
\begin{proof} It remains to show that $\phi$ does not vanish for a generic
metrisable structure. This stipulation is manifestly open on the suitably many
jets of a projective structure and it remains to check that there is at least 
one structure for which $\phi$ does not vanish identically. In other words, we 
need just one non-trivial example where our algorithm succeeds. Such an 
example is given in the following section.   
\end{proof}

\section{An example} In this section we carry out in detail our algorithm for
finding a metric in the projective class of a given connection. Consider the
torsion-free connection given in local co\"ordinates $(x^1,x^2,x^3)=(x,y,z)$ by
$$\nabla_aX^c=\partial_aX^c+\Gamma_{ab}{}^cX^b,$$
where $\partial_a\equiv\partial/\partial x^a$ and
$$\begin{array}{c}\Gamma_{ab}{}^1=\mbox{\small$\left[\!\begin{array}{ccc}
\displaystyle\frac{y+4xz+5x^2y}{(xy+z)(1+x^2)}
&\displaystyle\frac{x}{xy+z}
&\displaystyle\frac{xy+2z}{(xy+z)z}\\[14pt]
\displaystyle\frac{x}{xy+z}&0&0\\ [14pt]
\displaystyle\frac{xy+2z}{(xy+z)z}&0&0\end{array}\!\right]$},\\ \\
\Gamma_{ab}{}^2=\mbox{\small$\left[\!\begin{array}{ccc}
\displaystyle-(xy+z)xz^2
&\displaystyle\frac{2x}{1+x^2}&0\\[14pt]
\displaystyle\frac{2x}{1+x^2}&0&0\\ [14pt]
0&0&0\end{array}\!\right]$},\\ \\
\Gamma_{ab}{}^3=\mbox{\small$\left[\!\begin{array}{ccc}
\displaystyle-(xy+z)(xy+2z)z
&0&\displaystyle\frac{2x}{1+x^2}\\[14pt]
0&0&0\\ [14pt]
\displaystyle\frac{2x}{1+x^2}&0&0\end{array}\!\right]$}.\end{array}$$
If we take $\epsilon_{abc}$ to be the volume form with
$$\epsilon_{123}=(1+x^2)^4(xy+z)z,$$
then $\nabla_a\epsilon_{bcd}=0$ and a straightforward
computation gives
$$\begin{array}{c}V^{ab}{}_1=
\displaystyle\frac1{(1+x^2)^4}\mbox{\small$\left[\!\begin{array}{ccc}
0&0&0\\[4pt]
0&-2x&-2\\[4pt]
0&-2&2x\end{array}\!\right]$},\\ \\
V^{ab}{}_2=\displaystyle\frac1{(xy+z)^2(1+x^2)^4z^2}
\mbox{\small$\left[\!\begin{array}{ccc}
0&x&0\\[4pt]
x&0&0\\[4pt]
0&0&0\end{array}\!\right]$},\\ \\
V^{ab}{}_3=\displaystyle\frac1{(xy+z)^2(1+x^2)^4z^2}
\mbox{\small$\left[\!\begin{array}{ccc}
0&2&-x\\[4pt]
2&0&0\\[4pt]
-x&0&0\end{array}\!\right]$},\end{array}$$
for the Weyl tensor $V^{ab}{}_c$ defined by~(\ref{V}). One can check that
$Q_{ab}{}^c$ defined by (\ref{Q}) vanishes and therefore, in accordance with
Lemma~\ref{key}, we may write $V^{ab}{}_c=\rho^{d(a}\sigma^{b)e}\epsilon_{cde}$
for symmetric tensors $\rho^{ab}$ and $\sigma^{ab}$. Indeed, one can verify
that
$$\rho^{ab}=\frac1{(xy+z)^3(1+x^2)^4z^3}
\mbox{\small$\left[\begin{array}{ccc}
\displaystyle\frac4{(xy+z)^2z^2}&0&0\\[10pt] 
0&0&2x\\[4pt]
0&2x&4\end{array}\right]$}$$
and
$$\sigma^{ab}=\frac1{(1+x^2)^4}\mbox{\small$\left[\begin{array}{ccc}
1&0&0\\[4pt] 0&(xy+z)^2z^2&0\\[4pt] 0&0&(xy+z)^2z^2\end{array}\right]$}$$
will do. From the pencil~(\ref{pencil}), the combination 
$$\tilde\rho^{ab}=
\frac2{(xy+z)^2z^2}\sigma^{ab}
-\frac{(xy+z)^3z^3}2\rho^{ab}
=\frac1{(1+x^2)^4}\mbox{\small$\left[\begin{array}{ccc}
0&0&0\\[4pt] 0&2&-x\\[4pt] 0&-x&0\end{array}\right]$}$$
is singular and then 
$\tilde\sigma^{ab}=2(2+x^2)\sigma^{ab}-(xy+z)^5z^5\rho^{ab}$ gives
$$\tilde\sigma^{ab}
=\frac2{(1+x^2)^4}\mbox{\small$\left[\begin{array}{ccc}
x^2&0&0\\[4pt] 0&(2+x^2)(xy+z)^2z^2&-x(xy+z)^2z^2\\[4pt] 
0&-x(xy+z)^2z^2&x^2(xy+z)^2z^2\end{array}\right]$}$$
so that $\tilde\sigma_{ab}\tilde\rho^{ab}=0$. So now we have
$\tilde\sigma^{ab}$ and $\tilde\rho^{ab}$ from the pencil in the required
normal form (\ref{WLG}) and we may choose $\xi_a=[1,0,0]$ to compute the smooth
functions $\phi$ and $\psi$ from (\ref{phi_and_psi}). It turns out that
$$\phi=-\frac{4x^3}{(1+x^2)^9(xy+z)^2z^2}\quad\mbox{and}\quad
\psi=-\frac{8x^3}{(1+x^2)^9}$$
so then, according to~(\ref{purported}), if there is to be any solution of the 
metrisability equation~(\ref{metrisability}), then it must be of the form 
$h\hat\sigma^{ab}$, where   
$$\hat\sigma^{ab}=\tilde\sigma^{ab}-(\psi/\phi)\tilde\rho^{ab}
=\frac{2x^2}{(1+x^2)^4}
\mbox{\small$\left[\begin{array}{ccc}
1&0&0\\[4pt] 0&(xy+z)^2z^2&0\\[4pt] 0&0&(xy+z^2)z^2\end{array}\right]$}.$$
Theorem~\ref{scale} now applies and we should compute
$$\hat\sigma_{bc}(\nabla_a\hat\sigma^{bc})_\circ=
5\left(\frac{2xy+z-x^2z}{x(1+x^2)(xy+z)}\,dx
+\frac{x}{xy+z}\,dy+\frac{xy+2z}{(xy+z)z}\,dz
\right),$$
which is 
$$\nabla_a\big(5\log\Big(\frac{x(xy+z)z}{1+x^2}\Big)\big),$$
as required for a solution to~(\ref{metrisability}). Indeed, we conclude from 
(\ref{dlogh}) that, in order for $h\hat\sigma^{ab}$ to 
solve~(\ref{metrisability}), it must be that 
$$\textstyle5\nabla_ah+2h\hat\sigma_{bc}(\nabla_a\hat\sigma^{bc})_\circ=0,$$
whence
$$\nabla_a\log h+2\nabla_a\big(\log\Big(\frac{x(xy+z)z}{1+x^2}\Big)\big)=0$$
and so we may take
$$h=\left(\frac{1+x^2}{x(xy+z)z}\right)^2.$$
The upshot of all this is that, if there is to be a solution
to~(\ref{metrisability}), then up to an overall constant, it must be 
$$\sigma^{ab}
=\frac{1}{(1+x^2)^2}\mbox{\small$\left[\begin{array}{ccc}
\displaystyle\frac{1}{(xy+z)^2z^2}&0&0\\[10pt] 0&1&0\\[4pt]
0&0&1\end{array}\right]$}.$$
One easily verifies that this is, indeed, a solution from which it follows
that this example of a projective structure is metrisable. In fact, up to an
overall constant, the metric in question is
$$g_{ab}=\frac{\sigma_{ab}}{\det\sigma}=\frac{\sigma_{ab}}
{\epsilon_{cde}\epsilon_{pqr}\sigma^{cp}\sigma^{dq}\sigma^{dr}}
=\frac16\mbox{\small$\left[\begin{array}{ccc}
\displaystyle{(xy+z)^2z^2}&0&0\\ 0&1&0\\
0&0&1\end{array}\right]$}.$$

\appendix
\section{Normal forms}\label{normal}
Our aim here is to find a local normal form for a pair of elements from a
non-singular pencil $\{f\rho^{ab}+h\sigma^{ab}\}$ of symmetric contravariant
$2$-tensors on a smooth $3$-manifold, where {\em non-singular\/} means that one
element, say~$\sigma^{ab}$, from the pencil is non-singular. We are asking only
for preferred frames: it is a question of linear algebra concerning pencils of
symmetric matrices. For $n\times n$ matrices, suitable normal forms may be
found in~\cite{U}. In what follows we provide a more direct analysis for
$3\times 3$ matrices. As regards their application to the metrisability
problem, we shall need only the consequence that, generically, we may normalise
$\rho^{ab}$ and $\sigma^{ab}$ so that the conditions (\ref{WLG}) hold. More
precisely, we shall find that, generically and up to scale, there is exactly
one possible choice or exactly three possible choices for a degenerate
$\rho^{ab}$ from the pencil and, for each such choice, up to scale, just one
non-singular $\sigma^{ab}$ from the pencil such that $\sigma_{ab}\rho^{ab}=0$
and, up to scale, just one non-zero $\xi_a$ such that $\rho^{ab}\xi_a=0$.
Moreover, the construction is effective (since finding eigenvectors of a
$3\times 3$ matrix entails solving only a cubic polynomial) and the normalised
($\rho^{ab},\sigma^{ab},\xi_a$) may be chosen to depend smoothly on the base
manifold (since, generically, the roots of a polynomial depend smoothly on its
coefficients). For the remainder of this appendix we shall work with matrices
rather than tensors. In these terms, it is Lemma~\ref{generic_lemma} below that
is needed for the normalisation~(\ref{WLG}).

Suppose $H$ and $N$ are real symmetric $3\times 3$ matrices. If $H$ is
definite, then it is well-known that $N$ may be orthogonally diagonalised with
respect to~$H$. More precisely, it means that we may find 
$A\in{\mathrm{GL}}(3,{\mathbb{R}})$ such that
\begin{equation}\label{orthogonally_diagonalise}
A^tHA=\pm\mbox{\footnotesize$\left[\begin{array}{ccc}
1&0&0\\ 0&1&0\\ 0&0&1\end{array}\right]$}
\quad\mbox{and}\quad
A^tNA=\mbox{\footnotesize$\left[\begin{array}{ccc}
\lambda&0&0\\ 0&\mu&0\\ 0&0&\nu\end{array}\right]$}.\end{equation}
The following lemma deals with the case that $H$ is indefinite. 
\begin{lem}\label{indefinite_normal_forms}
Suppose $H$ and $N$ are real symmetric $3\times 3$ matrices with $H$
non-degenerate but indefinite. Then we may find
$A\in{\mathrm{GL}}(3,{\mathbb{R}})$ such that
$$A^tHA=\pm\mbox{\footnotesize$\left[\begin{array}{ccc}
1&0&0\\ 0&1&0\\ 0&0&-1\end{array}\right]$}$$
and
$$\begin{array}{rcl}A^tNA&=&\mbox{\footnotesize$\left[\begin{array}{ccc}
\lambda&0&0\\ 0&\mu&0\\ 0&0&\nu\end{array}\right]$}
\quad\mbox{or}\quad
\mbox{\footnotesize$\left[\begin{array}{ccc}
\lambda&0&0\\ 0&\alpha&\beta\\ 0&\beta&-\alpha\end{array}\right]$}\\[18pt]
&&\mbox{or}\quad
\mbox{\footnotesize$\left[\begin{array}{ccc}
1+\lambda&0&-1\\0&\mu&0\\ -1&0&1-\lambda\end{array}\right]$}
\quad\mbox{or}\quad
\mbox{\footnotesize$\left[\begin{array}{ccc}
\lambda&1&0\\ 1&\lambda&-1\\ 0&-1&-\lambda\end{array}\right]$},
\end{array}$$
these four possibilities being mutually exclusive.
\end{lem}
\begin{proof} If we follow the usual proof in case that $H$ is definite, the
only breakdown occurs when an eigenvector $H^{-1}N$ turns out to be null. It is
also possible that two eigenvalues occur as a complex conjugate pair. Thus, the
generic normal form for $N$ is diagonal over~${\mathbb{C}}$, as listed first
and second. For the other two possible normal forms
$$\xi=\mbox{\footnotesize$\left[\begin{array}{c}
1\\ 0\\ 1\end{array}\right]$}$$
is an eigenvector of $H^{-1}N$ and is null. These two cases correspond to
the possible non-diagonal Jordan canonical forms for $H^{-1}N$. Details are
left to the reader.
\end{proof}
Let us now consider a pencil $\Pi\equiv\{sN+tH\}$ of real symmetric $3\times 3$ 
matrices. Following~\cite{U}, we shall say that such a pencil is 
{\em non-singular\/} if one of its elements is non-singular. Without loss of
generality, let us suppose that $H$ is non-singular and consider the
homogeneous cubic polynomial
$${\mathbb{C}}^2\ni(s,t)\mapsto\chi(s,t)=\det(sH^{-1}N+t\,{\mathrm{Id}}).$$
Notice that $\chi(s,t)\not=0$ if and only if $sN+tH$ is non-singular.
Therefore, to require that $\chi(s,t)$ has three distinct zeroes on
${\mathbb{CP}}_1$ is independent of choice of~$H$. It is a property only of 
the pencil~$\Pi$ and we shall refer to such pencils as {\em regular\/}. 

\begin{lem}\label{generic_lemma}
Suppose $\Pi\equiv\{sN+tH\}$ is a non-singular regular pencil of real symmetric
$3\times 3$ matrices. Then we may find $N\in\Pi$ and
$\xi\in{\mathbb{R}}^3\setminus\{0\}$ such that $N\xi=0$. Each such $N$
determines, uniquely up to scale, a non-singular $H\in\Pi$ such that\/
${\mathrm{trace}}(H^{-1}N)=0$. This normal form $(N,\xi,H)$ can be arranged to
depend smoothly on the pencil\/~$\Pi$.
\end{lem}
\begin{proof}Whether $\Pi$ is regular or not, it is clear from the normal forms
(\ref{orthogonally_diagonalise}) and in Lemma~\ref{indefinite_normal_forms}
that, after a change of basis determined by the matrix~$A$, we may replace $N$
by $N\mp\lambda H$ so that, without loss of generality, these normal forms are
achieved with the $+$ sign for $H$ and with $\lambda=0$ in~$N$. Calculating
$\chi(s,t)$ in each case, we find the following cubic polynomials
$$\begin{array}{c}
t(t+\mu s)(t+\nu s)\quad\mbox{or}\quad t(t+\mu s)(t-\nu s)
\quad\mbox{or}\quad t[(t+\alpha s)^2+\beta^2s^2)\\[4pt]
\mbox{or}\quad t^2(t+\mu s)\quad\mbox{or}\quad t^3
\end{array}$$
Regularity of the pencil immediately eliminates the last two possibilities and
we are reduced to the following normal forms:
$$H=\mbox{\footnotesize$\left[\begin{array}{ccc}
1&0&0\\ 0&1&0\\ 0&0&1\end{array}\right]$}
\mbox{ and }
N=\mbox{\footnotesize$\left[\begin{array}{ccc}
0&0&0\\ 0&\mu&0\\ 0&0&\nu\end{array}\right]$}$$
in case $H$ is definite, and 
$$H=\mbox{\footnotesize$\left[\begin{array}{ccc}
1&0&0\\ 0&1&0\\ 0&0&-1\end{array}\!\right]$}
\mbox{ and }
N=\mbox{\footnotesize$\left[\begin{array}{ccc}
0&0&0\\ 0&\mu&0\\ 0&0&\nu\end{array}\right]$}
\mbox{ or }\mbox{\footnotesize$\left[\begin{array}{ccc}
\lambda&0&0\\ 0&\mu&0\\ 0&0&0\end{array}\right]$}
\mbox{ or }
\mbox{\footnotesize$\left[\begin{array}{ccc}
0&0&0\\ 0&\alpha&\beta\\ 0&\beta&-\alpha\end{array}\!\right]$}$$
in case $H$ is indefinite. In each case, the equation 
$${\mathrm{trace}}\big((H-tN)^{-1}N\big)=0$$
has a unique solution
for $t$, namely
$$t=\frac{\mu+\nu}{2\mu\nu}\quad\mbox{or}\quad
t=\frac{\nu-\mu}{2\mu\nu}\quad\mbox{or}\quad
t=\frac{\lambda+\mu}{2\lambda\mu}\quad\mbox{or}\quad
t=\frac{\alpha}{\alpha^2+\beta^2},$$
respectively. Regularity of the pencil ensures that the numerator in these
expressions is non-zero. Also, in this case, smooth dependence of $(N,\xi,H)$
on the pencil is clear from smooth dependence of the normal forms because
regularity means that the eigenvalues of $H^{-1}N$ are distinct and hence
depend smoothly on the pair $(N,H)$.
\end{proof}

\end{document}